\documentclass[pdflatex,sn-mathphys-num]{sn-jnl}

\usepackage{graphicx}%
\usepackage{multirow}%
\usepackage{amsmath,amssymb,amsfonts}%
\usepackage{amsthm}%
\usepackage{mathrsfs}%
\usepackage[title]{appendix}%
\usepackage{xcolor}%
\usepackage{textcomp}%
\usepackage{manyfoot}%
\usepackage{booktabs}%
\usepackage{algorithm}%
\usepackage{algorithmicx}%
\usepackage{algpseudocode}%
\usepackage{listings}%

\usepackage{mathtools}
\usepackage[font=small,labelfont=bf]{caption}
\usepackage[subrefformat=parens]{subcaption}
\numberwithin{equation}{section}
\usepackage{tikz}
\usepackage{bbm}
\usepackage{enumitem}

\newtheorem{theorem}{Theorem}[section]

\newtheorem{proposition}{Proposition}[section]
\newtheorem{corollary}{Corollary}[section]

\newtheorem{lemma}{Lemma}[section]

\theoremstyle{definition}
\newtheorem{definition}[theorem]{Definition}

\newtheorem{example}[theorem]{Example}

\theoremstyle{remark}
\newtheorem{remark}[theorem]{Remark}

\newcommand{\resmes}{\mathbin{\vrule height 1.6ex depth 0pt width
0.13ex\vrule height 0.13ex depth 0pt width 1.3ex}}

\DeclareMathOperator{\diam}{diam}

\def \Z {\mathbb{Z}}
\def \T {\mathbb{T}}

\def \R {\mathbb{R}}

\def \M {\mathcal{M}}
\def \P {\mathbb{P}}

\def \proj {\operatorname{proj}}

\DeclareMathOperator*{\supp}{supp}
\DeclareMathOperator*{\sign}{sign}

\begin{document}
\title[Subharmonic Kernels and Energy Minimizing Measures, with Applications to the Flat Torus]{Subharmonic Kernels and Energy Minimizing Measures, with Applications to the Flat Torus}

\author*[1]{\fnm{S.B.} \sur{Damelin}}\email{steve.damelin@gmail.com}

\author[2]{\fnm{Joel} \sur{Nathe}}\email{jnathe@purdue.edu}

\affil*[1]{\orgdiv{Department of Mathematics}, \orgname{Leibniz Institute for Information Infrastructure- FIZ Karlsruhe}, \orgaddress{\street{Franklinstr. 11}, \city{Berlin}, \postcode{10587}, \country{Germany}}}

\affil[2]{\orgdiv{Department of Mathematics}, \orgname{Purdue University}, \orgaddress{\street{150 N. University Street}, \city{West Lafayette}, \postcode{47907-2067}, \state{IN}, \country{USA}}}


\abstract{
    We study the minimization of the energy integral $I_K(\mu) = \int_{\Omega} \int_{\Omega} K(x,y) d\mu(x) d\mu(y)$ over all Borel probability measures $\mu$, where $(\Omega,\rho)$ is a compact connected metric space and $K:\Omega^2 \to [0,\infty]$ is continuous in the extended sense. We focus on kernels $K$ which are subharmonic, which we define 
    so that the potential $U_K^\mu(x) = \int_{\Omega} K(x,y) d\mu(y)$ satisfies a maximum principle on $\Omega \setminus \supp{\mu}$. This extends the classical electrostatics minimization problem for logarithmic energy $\int_{\Omega}\int_{\Omega}\log\left(\frac{1}{||x-y||}\right)$, which is used heavily as a tool in approximation theory.
    Using properties of minimizing measures, we show that if the singularities of the subharmonic kernel $K$ are such that $K$ is regular, then $K$ is positive definite, and $\mu$ is a minimizing measure if and only if its potential is constant (outside of a small exceptional set).
    \par
    We then apply this result to group invariant kernels on compact homogeneous manifolds. In this case, the uniform measure $\sigma$ has constant potential, so subharmonicity implies that this is the minimizing measure. Finally, we look at the case of the $d$-dimensional flat torus $\T^d$. We use our results to see that the Riesz kernel $K_s(x,y) = \sign(s)\rho(x,y)^{-s}$ is minimized by $\sigma$ (and thus positive definite) when $d > s \geq d-2$. Additionally, the positive definiteness gives us a condition which implies that the multivariate Fourier series of a function $f:[0,\pi]^d \to [0,\infty]$ has nonnegative coefficients.
}    


\maketitle

\section{Introduction}
In 1886, G. Robin \cite{Robin} proposed the following problem. For a compact $\Omega \subset \R^d$ ($d\geq1$), find a Borel probability measure $\mu$ supported on the boundary of $\Omega$ which generates a constant Newtonian potential $\int_{\Omega}K(x,y)d\mu(y)$, where $K:\Omega^2 \to [0,\infty]$ is the Newtonian kernel. 
It is known that such a measure $\mu$ also minimizes the energy integral
\[
\ \int_{\Omega}\int_\Omega K(x,y) d\mu(x)d\mu(y)
\]
among all Borel probability measures supported in $\Omega$. As we are in $\R^d$, We can write the Newtonian potential $K(x,y)$ as $K_s(x,y)$, where $K_s(x,y)$ is the Riesz kernel (see Example \ref{ex.RieszEx}) with $s = d-2$. It is also known that this minimizing measure $\mu$ exists as long as the energy integral above is finite for some measure.
\par
However, such an equivalence does not hold when $s \neq d-2$. When $d > s > d-2$, the support of the minimizing measure is not limited to $\partial \Omega$\textemdash in fact it must be all of $\Omega$. And when $s < d-2$, the potential of the minimizing measure may not be constant. 
\par
As an example, consider the (closed) unit disk in $\R^2$. Our Newtonian potential is $K_0(x,y) = \log \frac{1}{||x-y||}$, and the solution to the Robin problem is the measure uniform on the boundary $\mathbb{S}^1$. When $-2 < s < 0$, this remains the minimizing measure, but the potential is no longer constant. And when $0 < s < 2$, the minimizing measure is absolutely continuous with density $\rho(x) = C(1-||x||^2)^{s/2-1}$\cite{Landkof}.
\par
What is special about the Newtonian potential ($s=d-2$) is that it is harmonic outside of its support. In this paper, we study kernels that generate potentials which are subharmonic outside of their support, in the general setting of a compact connected metric space. We show that these kernels, which are analogous to the case $s \in [d-2,d)$, lead to the equivalence of constant potential and minimal energy. We also show an analogy to the potential's strict subharmonicity when $s > d-2$ forcing the minimizer to have full support.
\par
 We apply this to the flat torus with its geodesic metric and associated Riesz kernel, which is isometric to $(\R/2\pi \Z)^d$. The flat torus behaves similarly to a compact $\Omega \subset \R^d$, in the sense that for $d > s \geq d-2$, minimizing measures are those which have constant potential. And when $d > s > d-2$, all minimizing measures must have full support. We note that the uniform surface measure $\sigma$ satisfies these properties, and thus is a minimizer when $d > s \geq d-2$. The detailed statement of this result is Theorem \ref{thm.torrieszequil}.

\begin{remark}
The framework which we are developing in this paper allows for an extension of the classical  electrostatics logarithmic Newtonian energy problem for compact conductors. Let $\Omega\subset \mathbb C$ be a compact \emph{conductor} and consider the equilibrium distribution of unit charges of electrons on $\Omega$. Let the electrons repel each other under a logarithmic Newtonian (logarithmic inverse distance law). Equilibrium will be reached when the total logarithmic energy
\[
\int_{\Omega}\int_{\Omega}\log\left(\frac{1}{||x-y||}\right)d\mu(y)d\mu(x).
\]
is minimal among all charges.
Imagine now that the set $\Omega$ in the electrostatic problem above is a compact connected metric space and the logarithmic kernel is replaced by a kernel $K(x,y)$ which is symmetric ($K(x,y) = K(y,x)$) and continuous in the extended sense (see \ref{subsec.EngMinBackground}). Some examples of these spaces include the tori $\T^d$, spheres $\mathbb{S}^d$, ellipsoids, and other smooth manifolds.
Can one say something?

 
In approximation theory, one uses the classical electrostatic principle for logarithmic energy  as a fundamental tool for the approximation of real-valued functions by polynomials, objects intrinsically related to 
logarithmic kernels. Examples include Bernstein's density theorem in weighted approximation and Restricted range inequalities. We believe our framework should allow for extensions of these results and others where one should approximate real valued functions on general $\Omega$ by approximants other than polynomials.
\end{remark}
Some extension of this approximation theory has been done by Mhaskar in \cite{M1}, which extends many concepts of potential theory, including pre-capacity, pre-transfinite diameter, and the pre-Chebyshev constant to the case of a locally compact separable Hausdorff space. Mhaskar also extends properties for the minimizing measure, including the bounds we reference in Proposition \ref{prop.frost}. We will extend subharmonicity, a specific property of the logarithmic energy,to a general case and show that these kernels have many desirable properties.

Our paper is also motivated by the work of Damelin and his collaborators in the papers \cite{D,D1,DLS,DLRS,DHRZ}, and the classical work of Mhaskar \cite{M,M1} and Lubinsky \cite{Lub19a}, as well as the work of Bilyk and collaborators in \cite{BD,BDM,BMN,BMO}. In the 
papers \cite{D,D1,DLS,DLRS,DHRZ} and \cite{BD,BDM,BMN,BMO}, the following problems were studied:

\begin{itemize}
\item[(a)] The celebrated Gaussian quadrature formula on finite intervals tells us that the Gauss nodes are the zeros of the unique solution of an extremal problem. Damelin, Grabner \cite{DG} and Damelin, Levesley, Ragozin and Sun \cite{DLS,DLRS} derived quadrature estimates on compact, homogeneous manifolds embedded in Euclidean spaces, via energy functionals associated with a class of group-invariant kernels which are generalizations of zonal kernels on the spheres or radial kernels in Euclidean spaces. These results apply, in particular, to weighted Riesz kernels defined on spheres and certain projective spaces. The energy functionals describe both uniform and perturbed uniform distribution of quadrature point sets. 

\item[(b)] Let now $\Omega$ be the orbit of a compact, possibly non Abelian group $G$ acting as measurable transformations of $\Omega$  and the kernel $K:\Omega^{2} \rightarrow \mathbb R$ is invariant under the group action. The results of Damelin, Hickernell, Ragozin and Zeng \cite{DHRZ} show that for a natural minimal energy problem as in the above,
the unique minimizer $\mu_K$ is the normalized measure on $\Omega$ induced by Haar measure on $G$ which allows for explicit representations of $\mu_K$.
\item[(c)] In the two-point homogeneous spaces (the spheres $\mathbb{S}^d$ and the projective spaces $\mathbb{RP}^d,\mathbb{CP}^d$,$\mathbb{HP}^d$, and $\mathbb{OP}^1$,$\mathbb{OP}^2$), 
there are cases in which the energy minimizers for Riesz kernels (see Example \ref{ex.RieszEx}) are known precisely. In general, when $s$ is small (very negative), these minimizers are discrete, and have been shown to be minimizers  using linear programming or discrepancy arguments \cite{BDM,BMN}. 
When $s$ is large, the uniform surface measure (Haar measure) is the minimizer, which is shown by expanding the Riesz kernel into a series of positive definite (see Definition \ref{def.posdef}) kernels \cite{BD,BMN}.
\end{itemize}

A pattern in these above results (a-c) is that there are many cases where a uniform measure is a minimizer, and that this corresponds to the positive definiteness of the kernel (see Definition \ref{def.posdef}). Generally, these use the positive definiteness to show that the uniform measure is a minimizer. We will instead use subharmonicity of the kernel to show that uniform measures are minimizers, and also use it to obtain positive definiteness.

\subsection{Energy minimization background}
\label{subsec.EngMinBackground}
From now on, let $(\Omega,\rho)$ be a compact connected metric space, and our kernel $K: \Omega \times \Omega \to [0,+\infty]$ be symmetric and continuous in the extended sense. The most commonly studied kernels of this type (as they are generalizations of the logarithmic and Newtonian kernels) are the Riesz kernels. 

\begin{example}[Riesz Kernel] \label{ex.RieszEx}
The Riesz kernels (named for M. Riesz) are those of the following form:
\begin{equation} \label{Rieszker}
K_{s}(x,y) = 
\begin{cases}
\sign(s)\rho(x,y)^{-s}, & s \ne 0,\\
-\log(\rho(x,y)), & s=0.
\end{cases}
\end{equation}
These are commonly defined for the Euclidean distance on $\R^d$(ex.\cite{Landkof}), but we will use them in the setting of a general metric space. 
\par
In order to ensure that these kernels are always nonnegative, we may add a constant to them. This does not change which measures are minimizers, and is always possible, since $\Omega$ is compact and therefore bounded.
\end{example}

We denote the set of finite signed Borel measures on $\Omega$ as $\M(\Omega)$. Within this, we denote the set of finite unsigned Borel measures as $\M^+(\Omega)$, and the set of Borel probability measures (which are unsigned) as $\P(\Omega)$. We also denote the restriction of a measure $\mu$ to a set $S \subseteq \Omega$ by $\mu \resmes S$, and the support of $\mu$ by $\supp{\mu}$. As a final piece of notation, we write the closed ball centered at $x$ of radius $r$ as $B(x,r)$.
\par
Given a measure $\mu \in \M(\Omega)$, we define its \emph{$K$-potential} (or just \emph{potential}) $U_K^\mu: \Omega \to \R \cup \{-\infty,\infty\}$ by
\[
U_K^\mu(x) = \int_{\Omega} K(x,y) d\mu(y).
\]
From an energy perspective, the function $K(\cdot,y)$ is the \emph{potential field} induced by a unit point charge placed at $y$. So $U_K^\mu$ is the potential field induced by the measure $\mu$. We can then view $U_K^\mu(x)$ as then the \emph{potential energy} of a unit test charge placed at $x$ under this field.
\par
Thus, for $\mu \in \M^+(\Omega)$, we define the measure's \emph{$K$-energy} as 
\[
I_K(\mu) = \int_{\Omega} U_K^\mu(x) d\mu(x) = \int_{\Omega} \int_{\Omega} K(x,y) d\mu(x) d\mu(y).
\]
When $\mu$ is signed, we use Jordan decomposition to write $\mu = \mu^+ - \mu^-$, and define
\[
I_K(\mu) = I_K(\mu^+)+I_K(\mu^-) - 2\int_{\Omega}\int_{\Omega} K(x,y) d\mu^+(x) d\mu^-(y)
\]
as long as at least one of $I_K(\mu^+)+I_K(\mu^-)$ or $\int_{\Omega}\int_{\Omega} K(x,y) d\mu^+(x) d\mu^-(y)$ is finite.
\par
The generalization of the electrostatics problem here is minimizing the $K$-energy over measures in $\P(\Omega)$.
While we are more interested in finding the measures which are minimizers, the minimal energy itself is generally expressed as the \emph{$K$-capacity} of $\Omega$, which is defined as $c_K(\Omega) = \left(\min_{\mu \in \P(\Omega)} I_K(\mu)\right)^{-1}$. We note that capacity can be defined for any closed $F \subseteq \Omega$ by considering the metric space $(F,\rho)$. For other Borel sets $E$, we use the inner capacity $c_K(E) = \sup_{\substack{F \subseteq E \\ F \text{closed}}} c_K(F)$.
\par
If $c_K(\Omega) = 0$, then we have $I_K(\mu) = \infty$ for all $\mu \in \mathbb{P}(\Omega)$. As this makes minimizing the energy trivial, we assume in this paper that $c_K(\Omega) > 0$. Moreover, we assume $c(B(x,r)) > 0$ for all $r > 0$. This ensures that every nonempty open set has positive capacity, and is thus "included" in our minimization problem.
\par
Measures which achieve the minimum energy are called \emph{minimizing measures}. It is known that these measures exist and that their potentials have certain properties (see Section \ref{sec.equil}). One goal here, as before, is to explicitly determine or find properties of minimizing measures for a given kernel or class of kernels. We will show that for our class of subharmonic kernels, the minimizing measures must have constant potential (except possibly on an exceptional set).
\par
A important property for energy minimization is that of being positive definite.
\begin{definition}\label{def.posdef}
    We call a kernel $K$ \emph{positive definite} on $\Omega$ if
    \[
    I_K(\mu) \geq 0
    \]
    for any $\mu \in \M(\Omega)$ with $I_K(\mu)$ well-defined.
    We say that $K$ is \emph{strictly positive definite} if equality is only achieved when $\mu \equiv 0$.
    \par
    If the above inequality holds when we also require $\mu(\Omega)=0$, we say that $K$ is \emph{conditionally (strictly) positive definite} on $\Omega$. Clearly, (strict) positive definiteness implies (strict) conditional positive definiteness.
\end{definition}

 One example of the relationship between this concept and energy minimization is that if $K$ is strictly conditionally positive definite on $\Omega$, then the minimizing measure is unique (see, for example \cite[Ch.4]{BHS}).
Another application is on compact homogeneous manifolds with $G$-invariant kernels\textemdash see \S\ref{sec.torus}. In this case, positive definiteness (along with certain integrability conditions) implies that the normalized surface measure is the minimizer \cite{DLRS}.

One example of a strictly conditionally positive definite kernel is the Riesz kernel for $-2 \leq s < d$ on $\Omega \subset\mathbb{R}^d$ with the Euclidean distance \cite[Ch.4]{BHS}. Another is the Riesz kernel for $-1\leq s < d$ on the sphere $\mathbb{S}^d$\cite{BD}. We note that in both cases, the capacity of $\Omega$ is $0$ for $s \geq d$. 

\subsection{Subharmonic kernels and overview of the paper}

\par
We study a class of kernels that we call \emph{entirely subharmonic}. This property is based on the submean value inequality, with the mean defined with respect to a measure $\sigma \in \P(\Omega)$ which has full support.
\par
When the singularities of these entirely subharmonic kernels are sufficiently well-behaved (or non-existent), we show that any minimizing measure must have constant potential, outside of a small exceptional set. We define measures which have this property for a kernel $K$ as approximately $K$-invariant (Definition \ref{def.approxinvar}), and show that the reverse implication is also true when measures have finite energy. We also show that if the kernels are what we define as entirely strictly subharmonic (the submean value inequality is strict), then the minimizing measures must have full support.
\par
In Section \ref{sec.subharmonic}, we give our definition for subharmonicity of a function from $\Omega \to \R$, which is based on the submean value property, and show that the maximum principle applies to these functions. Finally, we give requirements for $U_K^\mu$ to be subharmonic on $\Omega \setminus \supp{\mu}$, a property which we call entirely subharmonic.
\par
In Section \ref{sec.subpotential} we work through difficulties involving the singularities at $K(x,x)$ affecting continuity at the boundary of $\supp \mu$. This section culminates with Theorem \ref{thm.firstmax}, the first maximum principle, which says that maximum of the potential must be achieved on the support. While this is known for Riesz kernels on $\R^d$(\cite{Landkof}), our version is set on the general metric measure space of $(\Omega,\rho,\sigma)$, for general subharmonic kernels.
\par
In Section \ref{sec.equil}, we apply our result from Section \ref{sec.subpotential} to minimizing measures. After introducing some well-known properties of minimizing measures, we prove the main theorem of the paper: For the class of kernels where Theorem \ref{thm.firstmax} applies, a measure is a minimizer if and only if it has constant potential (outside of a small exceptional set).
\par
In the last section, we apply our results from Section \ref{sec.equil} to the torus. We obtain new results which give a range of the parameter $s$ for which the uniform surface measure $\sigma$ is a minimizer for the Riesz energy. Additionally, we use the close relationship between minimization by the uniform surface measure $\sigma$, positive definiteness, and harmonic polynomials to generalize a result for the Fourier coefficients of convex functions.

\begin{remark}
Many of our results are significantly easier or more general in the case where the kernel $K$ is bounded. These cases are discussed in separate remarks or mentioned in the proofs of certain statements.
\end{remark}

\section{Subharmonic Kernels}
\label{sec.subharmonic}
We define subharmonic functions on metric measure spaces, and from that define subharmonic kernels.  Our definitions are based on the submean value property, chosen in order to apply the maximum principle. As was mentioned in the introduction, the measure $\sigma \in \mathbb{P}(\Omega)$ that we use to take the mean is required to have full support. Our subharmonic kernels are then defined in order to have potentials which are subharmonic functions. 
\par
\begin{definition}\label{def.subharmonic}
    Let $U\subseteq\Omega$ be open. We say an upper semi-continuous function $u: U \to [-\infty,\infty)$ is \emph{subharmonic with radius $R$} (or just \emph{subharmonic}) on $U$ if there exists $R > 0$ such that for every $B(x,r) \subseteq U$ with $0 < r < R$,
    \[
        u(x) \leq \frac{1}{\sigma(B(x,r))}\int_{B(x,r)} u(y) d\sigma(y)
    \]
    We say that $u$ is \emph{strictly subharmonic} on $U$ if the above inequality is always strict.
\end{definition}
When $\Omega \subseteq \R^d$ with the Euclidean norm, we recover the usual subharmonic functions if we take $\sigma$ to be the normalized Lebesgue measure.
\par 
Harmonic and subharmonic functions based on the mean value property have previously been discussed in \cite{GG} and \cite{AGG}. The difference between our definition and that of \cite{AGG} is that while we require upper semi-continuity, we only require the submean value property on balls smaller than a certain radius.
\par
We note that the subharmonicity of a function is dependent on the choice of $\sigma$. While this means that our class of subharmonic functions is not intrinsic to the metric space, it allows $\sigma$ to be chosen in order to make certain functions subharmonic.
\par
We now prove the maximum principle in the usual fashion. We note that while it is a requirement of Definition \ref{def.subharmonic}, the proof does not require $R$ to be independent of $x$. This uniformity condition is instead an extra property of subharmonic potentials obtained from the proof of Theorem \ref{thm.ufsubharm}. 
\medskip
\begin{proposition}[Maximum Principle]\label{prop.max}
    Assume $u$ is subharmonic on an open connected $U \subseteq \Omega$. Then if $u$ achieves its maximum at any point in $U$, $u$ is constant on $U$.
\end{proposition}

\begin{proof}
    Let $M = \sup_{x \in U} u(x)$. If $M$ is not finite this is trivial. Otherwise, we show that $V = \{x \in U| u(x) = M\}$ must be both open and closed in $U$. Since $u$ is upper semi-continuous, $U \setminus V = U \cap \{x \in U | u(x) < M\}$ is open, so $V$ is closed in $U$.
    \par
    Now consider $x_0 \in V$. Since $u$ is subharmonic and $U$ is open, there exists $0 < r < R$ such that $B(x_0,r) \subseteq U$ and 
    \begin{align*}
   M = u(x_0) &\leq \frac{1}{\sigma(B(x_0,r))} \int_{B(x_0,r)} u(y) d\sigma(y)\\
   0 &\leq \frac{1}{\sigma(B(x_0,r))}\int_{B(x_0,r)} u(y)-M d\sigma(y)\\
    \end{align*}
    Since $u$ is upper semi-continuous and $\sigma$ has full support, this implies that $u(y) = M$ for all $y \in B(x_0,r)$. Thus, $V$ is open in $U$.
    \par
    Since $U$ is connected and $V$ is both closed and open, $V = \emptyset$ or $V = U$. So if $u$ achieves its maximum on $u$, $V = U$, and $u$ is constant on $U$.
\end{proof}
In fact, if $u$ is upper semi-continuous on all of $\Omega$, it must achieve its maximum off of $U$. 
\medskip
\begin{corollary}\label{cor.max}
    If $u:\Omega \to [0,\infty]$ is upper semi-continuous on $\Omega$ and subharmonic on an open $U \subsetneq \Omega$, then $u$ achieves its maximum at some point $y \notin U$.
\end{corollary}
\begin{proof}
     Let $M = \max_{x \in U} u(x)$, and consider $E = \{y\in\Omega| u(y) = M\}$. Since $u$ is upper semi-continuous and $\Omega$ compact, this is closed and nonempty. If $E \not\subseteq U$, we are done. 
    \par
    Otherwise, $u$ achieves its maximum on $U$. Without loss of generality, let $U$ be connected, since we can choose a connected component of $U$ where $u$ achieves its maximum. In this case, $E = U$. Thus, $U$ is both closed and open. This implies that $U = \Omega$, a contradiction.
\end{proof}
Note that if $u$ is strictly subharmonic, then $u$ cannot be constant on $U \neq \emptyset$, and thus cannot achieve its maximum on $U$. 
We are now ready to define our subharmonic kernels.
\medskip
\begin{definition}
    Let $U \subseteq \Omega$ be open and $F \subseteq \Omega$ be closed. Then a kernel $K$ is (strictly) subharmonic on $U \times F$ if there exists $R > 0$ such that for each $y \in F$, $K(\cdot,y)$ is (strictly) subharmonic on $U$. That is, $K(\cdot,y):U \to [-\infty,\infty)$ is upper semi-continuous and for all $B(x_0,r) \subseteq U$ ($0 < r < R$),
    \[
    K(x_0,y) \leq \frac{1}{\sigma(B(x_0,r))} \int_{B(x_0,r)} K(z,y) d\sigma(z)
    \]
\end{definition}
We note that this definition requires the same choice of $R$ for all $y \in F$. This uniformity is necessary in order to ensure our kernels have subharmonic potentials. This is our novel idea, which we will use to attack the energy minimization problem.
\medskip
\begin{theorem}\label{thm.ufsubharm}
    If $K$ is (strictly) subharmonic on $U \times F$ and $0 \not\equiv \mu \in \mathcal{M}^+(\Omega)$ with $\supp{\mu} \subseteq F$, then $U_K^\mu$ is continuous, finite, and (strictly) subharmonic on $U$.
\end{theorem}
\begin{proof}
    We first prove continuity, remembering that we assume $K$ is continuous in the extended sense.\par
    For any $x_0 \in U$, consider $\emptyset \neq B(x_0,r) \subseteq U$. Then $K$ is continuous and bounded on the compact set $B(x_0,r) \times F$. Since $\mu(\Omega) <\infty$, by bounded convergence,
    \[
    \lim_{x\to x_0}U_K^\mu(x) = \lim_{x\to x_0} \int_{F} K(x,y) d\mu(y) = \int_{F} \lim_{x\to x_0} K(x,y) d\mu(y) = U_K^\mu(x_0).
    \]
    So $U_K^\mu$ is continuous and finite on $U$.
    \par
    Now for the submean value inequality, consider $B(x_0,r) \subseteq U$ with $0 < r < R$. Since $K$ is nonnegative, we can rearrange integrals to obtain 
    \begin{align*}
        \frac{1}{\sigma(B(x_0,r))}\int_{B(x_0,r)} U_K^\mu(x) d\sigma(x) &= \frac{1}{\sigma(B(x_0,r))}\int_{B(x_0,r)}  \int_{\Omega} K(x,y) d\mu(y) d\sigma(x)\\
        &= \int_\Omega \frac{1}{\sigma(B(x_0,r))} \int_{B(x_0,r)} K(x,y) d\sigma(x) d\mu(y)\\
        &\geq \int_{\Omega} K(x_0,y) d\mu(y) = U_K^\mu(x_0)
    \end{align*}
    So $U_K^\mu$ is subharmonic with radius $R$. And the inequality becomes strict for all $y$ when $K$ is strictly subharmonic, so in that case $U_K^\mu$ is strictly subharmonic as long as $\mu \not \equiv 0$.
\end{proof}
The above theorem requires $\supp{\mu} \subseteq F$. When we consider energy minimization across all $\mu \in \P(\Omega)$, this will not be true in general. What we want is kernels for which we can choose $F = \supp{\mu}$ and $U = \Omega \setminus \supp{\mu}$ for any $\mu$.
\medskip
\begin{definition}
    A kernel $K:\Omega \times \Omega \to [0,\infty]$ is \emph{entirely subharmonic} if there exists $R > 0$ such that for all $y \in \Omega$, $x \mapsto K(x,y)$ is finite and subharmonic with radius at least $R$ on $\Omega \setminus \{y\}$.
    \medskip
\end{definition}
It is clear that for all $\mu \in \mathcal{M}^+(\Omega)$, an entirely subharmonic $K$ is subharmonic on $(\Omega \setminus \supp{\mu}) \times \supp{\mu}$. From now on, if we are not referring to a specific $U \times F$, we may drop the \emph{entirely} and just call these kernels \emph{subharmonic}.

\section{Potentials for Subharmonic Kernels}\label{sec.subpotential}
    We wish to apply Corollary \ref{cor.max} to the potentials of our entirely subharmonic kernels. However, when $K$ is not finite (which is possible at $K(x,x)$), we cannot be sure that $U_K^\mu$ is upper semi-continuous on $\Omega$. But for some of these kernels (in particular Riesz kernels), we show that for any $\mu \in \P(\Omega)$, we can apply Corollary \ref{cor.max} to approximations of $\mu$, and then take the limit to achieve our results.
    \par
    We start with a basic property for all potentials, which is a simple consequence of Fatou's lemma.
    \medskip
    \begin{lemma}
        If $\mu \in \mathcal{M}^+(\Omega)$, then $U_K^\mu$ is lower semi-continuous.
    \end{lemma}
    \par
    We now shift focus to the following class of kernels, which allow our approximations.
    \medskip
    \begin{definition}
        A kernel $K:\Omega \times \Omega \to [0,\infty]$ is \emph{regular} if for any $\mu \in \M^+(\Omega)$ with $U_K^\mu$ continuous and finite on $\supp \mu$, $U_K^\mu$ is continuous and finite on $\Omega$.
    \end{definition}
    Aside from the upcoming argument, regularity is important for certain convergence results \cite[Ch.6\S2]{Landkof}. We are most interested in Riesz kernels, which are known to be in this class. We provide a proof.
    \medskip
    \begin{proposition}
        The Riesz kernels $K_s$ are regular.
    \end{proposition}
    
    \begin{proof}
        Consider $x_0 \in \Omega$. If $x_0 \notin \supp{\mu}$, then there exists $B(x_0,r) \subseteq \Omega \setminus \supp \mu$. Since $K_s$ is bounded on $B(x_0,r) \times \supp \mu$, we can use the argument in the proof of Theorem \ref{thm.ufsubharm} to show that $U_K^\mu$ is continuous and finite at $x_0$.
        \par
        Thus, we only need to consider $x_0 \in \supp{\mu}$. Define $\mu_r = \mu \resmes B(x_0,r)$, and $\nu_r = \mu - \mu_r$. Then for any $x \in \Omega$,
        \[
            U_K^\mu(x) = U_K^{\mu_r}(x) + U_K^{\nu_r}(x)
        \]
        and
        \[
        |U_K^\mu(x) - U_K^\mu(x_0)| \leq U_K^{\mu_r}(x) + U_K^{\mu_r}(x_0) + |U_K^{\nu_r}(x) - U_K^{\nu_r}(x_0)|
        \]
        We note that $U_K^{\nu_r}$ is continuous at $x_0$ (since $x_0 \notin \supp{\nu_r}$). And since $\mu_r = \mu - \nu_r$, $U_K^{\mu_r}$ is continuous on $\supp{\mu}$ at $x_0$.
        \par
        We now bound $U_K^{\mu_r}$ off its support by its values on the support. Let $\epsilon > 0$ be arbitrary.
        \par
        Since $U_{K_s}^\mu(x_0) < \infty$ and continuous on $\supp \mu$, we can choose $r > \delta > 0$ such that $U_{K_s}^{\mu_r}(x_0) < \epsilon$, $U_{K_s}^{\mu_r}(x) < 2\epsilon$ for $x \in \supp_\mu \cap B(x_0,\delta)$, and $|U_{K_s}^{\nu_r}(x) - U_{K_s}^{\nu_r}(x_0)|< \epsilon$ for all $x \in B(x_0,\delta)$.
        \par
        Consider $x \in B(x_0,\delta)$, and let $x'$ be a point in $\supp{\mu}$ closest to $x$. 
        For all $y \in \supp{\mu}$, $\rho(x',y) \leq \rho(x,x') + \rho(x,y) \leq 2\rho(x,y)$. So for $s > 0$, $K_s(x,y)\leq 2^s K_s(x',y)$. Thus,
        \begin{align*}
        U_K^{\mu_r}(x) &= \int_{B(x_0,r)} K(x,y) d\mu(y)\\
        &\leq \int_{B(x_0,r)} 2^s K(x',y) d\mu(y)\\
        &\leq 2^{s+1} \epsilon.\\
        \end{align*}
        Then
        \[
            |U_{K_s}^\mu(x) - U_{K_s}^\mu(x_0)| \leq 2^{s+1} \epsilon + \epsilon + \epsilon.
        \]
        When $s = 0$, we note that
        \[
        I_{K_0}(\mu) \geq I_{K_0}(\mu_r) \geq -\log (2r) \mu_r(\Omega)^2,
        \]
        so we can take our $r$ small enough so that $\mu_r(\Omega) < \epsilon$. Then $K_0(x,y) \leq \log 2 + K_0(x',y)$, so
        \begin{align*}
            U_K^{\mu_r}(x) &= \int_{B(x_0,r)} K(x,y) d\mu(y)\\
            &\leq \int_{B(x_0,r)} K(x',y) d\mu(y) + \epsilon\log 2\\
            &\leq 2\epsilon +\epsilon \log 2.
        \end{align*}
        Then 
        \[
        |U_{K_0}^\mu(x) - U_{K_0}^\mu(x_0)| \leq 2 \epsilon + \epsilon \log 2 + \epsilon + \epsilon.
        \]
        Since $\epsilon$ was arbitrary, we have shown that $U_{K_s}^\mu$ is continuous in both cases.
    \end{proof}
    Next, we use Lusin's theorem to show that for any $\mu$, there is an approximation to $\mu$ which is continuous on its support, and thus continuous on $\Omega$.
    \medskip
    \begin{proposition}
        Assume $K$ is regular and $U_K^\mu$ is finite. Then for any $\epsilon > 0$, there exists a compact set $F \subseteq \supp{\mu}$ such that $\mu' = \mu \resmes F$ has $U_K^{\mu'}$ is continuous on $\Omega$ and $\mu'(\Omega) = \mu(F) > 1-\epsilon$.
    \end{proposition}
    \begin{proof}
        Since $\mu$ is a Borel probability measure on a compact metric space, by Lusin's theorem, there exists compact $F \subseteq \supp{\mu}$ such that $U_K^\mu$ is continuous on $F$ and $\mu(\supp{\mu} \setminus F) < \epsilon$. So we need to show that $U_K^{\mu'}$ is continuous on $F$. But we have
        \[
            U_K^{\mu'}(x) = U_K^\mu(x) + (-U_K^{\mu-\mu'}(x))
        \]
        which is the sum of two upper semi-continuous functions, and is thus upper semi-continuous. And we know $U_K^{\mu'}$ is lower semicontinuous. Thus, $U_K^{\mu'}$ is continuous on $F = \supp{\mu'}$. So by the previous proposition, $U_K^{\mu'}$ is continuous on $\Omega$.
    \end{proof}
    \medskip
    Finally, we construct a monotonically increasing sequence from these approximations.
    \medskip
    \begin{proposition}\label{prop.incmeasure}
        If $K$ is regular and $U_K^\mu$ is finite on $\Omega$, there exists an increasing sequence of measures $(\mu_n)_{n=1}^\infty$ such that $U_K^{\mu_n}$ is continuous on $\Omega$ and $\lim_{n\to\infty} U_K^{\mu_n}(x) = U_K^\mu(x)$ for all $x \in \Omega$.
    \end{proposition}
    \begin{proof}
        We set $\nu_1 = 0$, and use the previous proposition to define $\nu_n$ inductively. For each $n > 1$, we choose $F_n \subseteq \supp (\mu - \nu_{n-1})$ with $(\mu-\nu_{n-1})(F_n) > (\mu-\nu_{n-1})(\Omega)-\frac{1}{n}$ such that $\nu_n = (\mu-\nu_{n-1}) \resmes F_n$ has $U_K^{\nu_n}$ continuous.
        \par
        If we set $F_1 = \emptyset$, we see that each $K_n$ is disjoint, and we have $\nu_n = \mu \resmes F_n$. 
        Then let 
        \[
        \mu_n = \sum_{k=1}^n \nu_k.
        \]
        Since $U_K^{\mu_n}$ is a finite sum of continuous potentials, it is continuous. And $\mu_n$ is clearly increasing to $\mu$. So we must have $\lim_{n\to\infty} U_K^{\mu_n}(x) = U_K^{\mu}(x)$.
    \end{proof}
    With these approximations, we are ready to prove the first maximum principle for any $\mu \in \M^+(\Omega)$.
    \medskip
    
    \begin{theorem}[First Maximum Principle]\label{thm.firstmax}
        Assume $K$ is regular and entirely subharmonic, and $\mu \in \M^+(\Omega)$. Then if $U_K^\mu(x) \leq M$ on $\supp{\mu}$, $U_K^\mu(x) \leq M$ on $\Omega$.
    \end{theorem}
    \begin{proof}
        If $U_K^\mu$ is not finite, then $M = \infty$ and we are done. Otherwise, let $\mu_n$ be a sequence of measures from Proposition \ref{prop.incmeasure}. We note that for each $\mu_n$, $U_{K}^{\mu_n} \leq U_K^\mu$ is subharmonic on $\Omega \setminus \supp{\mu_n} \supseteq \Omega \setminus \supp {\mu}$ and continuous on $\Omega$. 
        So by Corollary \ref{cor.max}, $U_K^{\mu_n}(x) \leq M$ for all $x \in \Omega$.  Thus, for all $x \in \Omega$, $U_K^\mu(x) = \lim_{n\to\infty} U_K^{\mu_n}(x) \leq M$.
    \end{proof}
    This result has previously been shown for Riesz kernels on $\R^d$\textemdash a proof using a similar argument can be found in \cite[Ch.1]{Landkof}. We have extended it here to general metric measure spaces, which allows us to use our general definition of subharmonic kernels.
    \medskip
    \begin{remark}\label{rmk.firstmax}
        If $K$ is finite, then a simple bounded convergence argument allows us to see that $U_K^\mu$ is continuous on all of $\Omega$. Combining this with the maximum principle easily yields Theorem \ref{thm.firstmax}.
    \end{remark}
\section{Minimizing Measures and Positive Definiteness}\label{sec.equil}
We turn our attention to energy minimization and minimizing measures. Combining Theorem \ref{thm.firstmax} with known properties of minimizing measures results in Theorem \ref{thm.kinvar}, which states that for regular subharmonic kernels, minimizing measures are precisely those measures with constant potential (outside of an exceptional set). 
\par
Minimization by these measures (which we call approximately $K$-invariant), is closely related to the kernel being positive definite (see Definition \ref{def.posdef}). In fact, a common method for showing that these measures are minimizers is to show that the kernel is positive definite.
\par
While the equivalence between these properties is well-known for finite $K$\cite{BMO}, when we allow $K$ to reach $+\infty$, it is only known that positive definiteness implies approximately $K$-invariant minimizers. In the second subsection we instead show that positive definiteness is a consequences of the first maximum theorem.

\subsection{Minimizing Measures}
We wish to apply Theorem \ref{thm.firstmax} to minimizing measures. We first show the existence of a minimizer, adapting a proof from \cite{Ransford}.
\medskip
\begin{proposition}\label{prop.equexist}
    There exists $\mu \in \P(\Omega)$ such that $I_K(\mu) = W_K(\Omega)$. 
\end{proposition}
\begin{proof}
    We can choose a sequence of measures $(\mu_n)_{n=1}^\infty$ such that $\lim_{n\to\infty} I_K(\mu_n) = \inf_{\mu \in \P(\Omega)} I_K(\mu)$. Then since $\Omega$ is compact, there exists a subsequence $(\mu_{n_i})_{i=1}^\infty$ which is weak-* convergent to some $\mu \in \P(\Omega)$ (written $\mu_{n_i} \overset{\ast}{\rightharpoonup} \mu$).
    Now for $m \in \R$, define $K_m(x,y) = \min \{K(x,y),m\}$, which is a continuous and finite kernel. Then $\mu_{n_i} \times \mu_{n_i} \overset{\ast}{\rightharpoonup} \mu \times \mu$, so for any $m > 0$ we have
    \begin{align*}
    \liminf_{i\to\infty} I_K(\mu_{n_i}) &=  \liminf_{i\to\infty} \int_{\Omega} \int_{\Omega} K(x,y) d\mu_{n_i}(x) d\mu_{n_i}(y)\\
    &\geq \liminf_{i\to\infty} \int_{\Omega} \int_{\Omega} K_m(x,y) d\mu_{n_i}(x) d\mu_{n_i}(y)\\
    &= \int_{\Omega} \int_{\Omega} K_m(x,y) d\mu(x) d\mu(y)\\
    &= I_{K_m}(\mu)
    \end{align*}
    Since $I_{K_m}(\mu) \to I_K(\mu)$ as $m\to \infty$ by monotone convergence, we have 
    \[
    I_K(\mu) = \lim_{m\to\infty} I_{K_m}(\mu) \leq \liminf_{i\to\infty} I_K(\mu_{n_i})= \inf_{\mu \in \P(\Omega)} I_K(\mu).
    \]
    Since $\mu \in \P(\Omega)$, we have $I_K(\mu) = \inf_{\mu \in \P(\Omega)} I_K(\mu)$.
\end{proof}
The first maximum principle gives a result for the potentials of measures. To apply it, we need properties for the potentials of minimizers. We will use the following property, which is classical\cite{BHS}. We provide a proof.
\medskip
\begin{proposition}\label{prop.frost}
    If $\mu$ is the minimizing measure for $K$ on $\Omega$, then $U_K^\mu(x) \leq I_K(\mu)$ for all $x \in \supp{\mu}$ and $U_K^\mu(x) \geq I_K(\mu)$ on $\Omega$ except possibly on a set of capacity 0.
\end{proposition}
\begin{proof}
Assume there exists $x_0 \in \supp{\mu}$ such that $U_K^\mu(x_0) > I_K(\mu)$. Then since $U_K^\mu$ is lower semi-continuous, there exists $r>0$ such that $U_K^\mu(x) > \frac{U_K^\mu(x_0)+I_K(\mu)}{2}$ for all $x \in B(x_0,r)$. Then consider
\[
    \nu = \frac{1}{1-\mu(B(x_0,r))}\mu\resmes(\Omega \setminus B(x_0,r)) = \frac{1}{1-\mu_r(\Omega)} (\mu-\mu_r)
\]
where $\mu_r =  \mu \resmes B(x_0,r)$.
We have
\begin{align*}
    I_K(\nu) &= \frac{1}{(1-\mu_r(\Omega))^2}\left(I_K(\mu) + I_K(\mu_r) - 2\int_{\Omega} \int_{\Omega} K(x,y) d\mu(y) d\mu_r(x)\right)\\
    &\leq \frac{1}{(1-\mu_r(\Omega)^2)^2}\left(I_K(\mu) - \int_{\Omega} U_K^\mu(x) d\mu_r(x)\right)\\
    &\leq \frac{1}{(1-\mu_r(\Omega))^2} (I_K(\mu)-\mu_r(\Omega) I_K(\mu))\\
    &= \frac{I_K(\mu)}{1-\mu_r(\Omega)}
\end{align*}
Since $x_0 \in \supp{\mu}$, $1-\mu_r(\Omega) < 1$, so this contradicts $\mu$ being a minimizing measure. Thus, $U_K^\mu(x) \leq I_K(\mu)$ for all $x \in \supp{\mu}$.
\par
For the other inequality, let $T_n = \{x \in \Omega | U_K^\mu(x) \leq I_K(\mu)-\frac{1}{n}\}$, which is compact by lower semi-continuity. If $T_n$ does not have capacity 0, there exists $\nu \in \P(\Omega)$ with $\supp \nu\subseteq T_n$ such that $I_K(\nu) < \infty$. 
Now, consider $\mu' = (1-\epsilon)\mu+\epsilon \nu$, where $0 < \epsilon \leq \frac{2}{2+nI_K(\nu)}$. Then we have
\begin{align*}
    I_K(\mu') &= (1-\epsilon)^2 I_K(\mu) + \epsilon^2I_K(\nu) + 2\epsilon(1-\epsilon) \int_{\Omega} U_K^\mu(x) d\nu(x)\\
    &\leq (1-\epsilon)^2 I_K(\mu) + \epsilon^2I_K(\nu) + 2\epsilon(1-\epsilon) (I_K(\mu) -\frac{1}{n})\\
    &= (1-\epsilon)I_K(\mu)(1-\epsilon+2\epsilon) + \epsilon^2I_K(\nu) -\frac{2}{n}\epsilon(1-\epsilon)\\
    &\leq (1-\epsilon^2)I_K(\mu)
\end{align*}
which is a contradiction. So $T_n$ must have capacity 0. Then $\{x \in \Omega | U_K^\mu(x) < I_K(\mu)\} = \bigcup_{n=1}^\infty T_n$ is our exceptional set. Since $\supp \nu \subseteq \bigcup_{n=1}^\infty T_n$ implies $\nu \resmes T_i \not \equiv 0$ for some $i$, this exceptional set has capacity 0.
\end{proof}
Applying the first maximum principle would tell us that $U_K^\mu \leq I_K(\mu)$ on all of $\Omega$. 
Then when $K$ is finite, all sets have positive capacity, and this implies that $U_K^\mu$ is constant on $\Omega$. When the potential is constant (and finite), $\mu$ is called $K$-invariant. For the more general case, we achieve measures with the following property.
\begin{definition}\label{def.approxinvar}
    We call a measure $\mu \in \P(\Omega)$ \emph{approximately $K$-invariant} if $I_K(\mu) < \infty$ and there exists $M \in \R$ such that $\{x \in \Omega | U_K^\mu(x) \neq M\}$ has capacity 0.
\end{definition}
We note that $\mu$ approximately $K$-invariant implies that $U_K^\mu = I_K(\mu) = M$ $\mu$-almost everywhere. So we can equate the energies of approximately $K$-invariant measures, and are ready to apply Theorem \ref{thm.firstmax}.
\begin{proposition}\label{prop.kinveq}
    If $\mu$ and $\nu$ are approximately $K$-invariant, then $I_K(\mu) = I_K(\nu)$.
\end{proposition}
\begin{proof}
    Let \[
    T \coloneqq\{U_K^\mu \neq I_K(\mu)\} \cup \{U_K^\nu \neq I_K(\nu)\},
    \] be the exceptional set. Then since $\mu(T) = \nu(T) = 0$,
    \[
    I_K(\mu) = \int_{\Omega} U_K^\mu(x) d\nu(x) = \int_{\Omega} U_K^\nu(y) d\mu(y) = I_K(\nu).
    \].
\end{proof}
\medskip
\begin{theorem}\label{thm.kinvar}
    If $K$ is regular and entirely subharmonic, then $\mu$ is a minimizer if and only if $\mu$ is approximately $K$-invariant.
\end{theorem}
\begin{proof}
    If $\mu$ is a minimizer, then Proposition \ref{prop.frost} tells us that $U_K^\mu(x) \leq I_K(\mu)$ on $\supp{\mu}$. Since $K$ is regular and entirely subharmonic, Theorem \ref{thm.firstmax} tells us that $U_K^\mu(x) \leq I_K(\mu)$ on $\Omega$. So $\{U_K^\mu(x) \neq I_K(\mu)\} = \{U_K^\mu(x) < I_K(\mu)\}$. The second part of Proposition \ref{prop.frost}, along with our assumption that $\Omega$ has positive $K$-capacity, implies that $\mu$ is approximately $K$-invariant.
    \par
    For the other direction, Proposition \ref{prop.equexist} tells us that there exists a minimizing measure $\nu$. The first direction tells us that $\nu$ is approximately $K$-invariant, so  Proposition \ref{prop.kinveq} implies that $I_K(\mu) = I_K(\nu) = W_K(\Omega)$. Thus, $\mu$ is a minimizer.
\end{proof}
\medskip
\begin{corollary}\label{thm.fullsup}
    If $K$ is regular and entirely strictly subharmonic, then any minimizer $\mu$ must have full support.
\end{corollary}
\begin{proof}
    Theorem \ref{thm.kinvar} tells us that $\mu$ is approximately $K$-invariant. Since $U_K^\mu$ is continuous on $\Omega \setminus \supp{\mu}$, and we have assumed that all nonempty balls have positive capacity,  $U_K^\mu$ is constant and strictly subharmonic on $\Omega \setminus \supp{\mu}$. But this is impossible unless the open set $\Omega \setminus \supp{\mu}$ is empty, so $\mu$ has full support.
\end{proof}
\subsection{Positive Definiteness and $K$-invariance}
As was mentioned in the introduction (see Definition \ref{def.posdef}), positive definiteness is closely related to energy minimization. Here, we show that in addition the having approximately $K$-invariant minimizers, our regular subharmonic kernels are positive definite. First, we give a simple result which shows that if there is a $K$-invariant measure, then we do not have to worry about the distinction between positive definiteness and conditional positive definiteness.
\begin{proposition}
    Assume there exists a $K$-invariant measure $\sigma$. Then $K$ is (strictly) conditionally positive definite if and only if $K$ is (strictly) positive definite.
\end{proposition}
\begin{proof}
    That (strict) positive definiteness implies (strict) conditional positive definiteness is trivial. Now, assume $\mu \in \mathcal{M}(\Omega)$. Then
    \begin{align*}
        I_K(\mu-\mu(\Omega)\sigma) &= I_K(\mu) + \mu(\Omega)^2 I_K(\sigma) - 2\mu(\Omega)^2I_K(\sigma)\\
        I_K(\mu) &= I_K(\mu-\mu(\Omega)\sigma) + \mu(\Omega)^2I_K(\sigma)
    \end{align*}
    Since $\mu - \mu(\Omega)\sigma$ has total mass 0 and $K$ being nonnegative implies $I_K(\sigma) \geq 0$, if $K$ is (strictly) conditionally positive definite, $K$ is (strictly) positive definite.
\end{proof}
\begin{remark}
    If $K$ is finite and there exists a $K$-invariant measure, positive definiteness is equivalent to many other conditions, including minimization by the $K$-invariant measure and convexity of the energy functional $I_K$. More details and a proof can be found in \cite{BMO}.
\end{remark}
For non-finite kernels, it is easy to show that positive definiteness implies that any approximately $K$-invariant measure $\mu$ is a minimizer (consider $I(\nu-\mu)$). However, the converse statement appears to be an open problem for general kernels.
\par
But for our regular subharmonic kernels, we are able to prove positive definiteness using a result by Ninomiya\cite[Thm.3]{Ninomiya}. 
\begin{theorem}[Ninomiya]\label{thm.fmposdef}
    If $K$ satisfies the first maximum principle, then $K$ is positive definite.
\end{theorem}
\begin{corollary}\label{cor.subposdef}
    All regular, entirely subharmonic kernels are positive definite.
\end{corollary}
As the original result by Ninomiya is in French, we will show the proof.
We first define the quantity 
    \[
    G(\mu,\nu)\coloneqq \frac{I_K(\mu)I_K(\nu)}{\left(\int_\Omega U_K^\mu(x) d\nu(x)\right)^2}
    \]
for $\mu,\nu \in \M^+(\Omega)$ with $I_K(\mu),I_K(\nu),\int_{\Omega} U_K^\mu(x) d\nu(x)$ positive and finite. Using the arithmetic-geometric mean inequality, we see that if $G(\mu,\nu) \geq 1$, $I_K(\mu-\nu) \geq 0$.
\par
For $E_1$ and $E_2$ disjoint compact subsets of $\Omega$, a similar argument to Proposition \ref{prop.equexist} (see \cite[Remark 1]{Ninomiya}) shows that there exists a minimizing pair $\mu_0,\nu_0$ with $\supp{\mu_0} \subseteq E_1$, $\supp{\nu_0} \subseteq E_2$ such that
    \[
    G(\mu_0,\nu_0) = \inf_{\substack{\supp{\mu} \subseteq E_1\\ \supp{\nu}\subseteq E_2}} G(\mu,\nu).
    \]
So we wish to show that $G(\mu_0,\nu_0) \geq 1$.
\medskip
\begin{lemma}
    Let $a = I(\mu_0)$, $b = I(\nu_0)$, and $c = \int_{\Omega} U_K^{\mu_0}(x) d\nu_0(x)$. Then $cU_K^{\mu_0}(x) \leq a U_K^{\nu_0}(x)$ $\mu_0$-almost everywhere in $E_1$, and $cU_K^{\nu_0}(x) \leq bU_K^{\mu_0}(x)$ $\nu_0$-almost everywhere in $E_2$
\end{lemma}
\begin{proof}
    By symmetry, we only need to show that $cU_K^{\mu_0}(x) \leq a U_K^{\nu_0}(x)$ $\mu_0$-almost everywhere in $E_1$. Let $g(x) = cU_K^{\mu_0}(x) - aU_K^{\nu_0}(x)$, and $\mu_0' = \mu_0 \resmes \{g > 0\}$. From the inequality
    \[
    G(\mu_0-\epsilon \mu_0',\nu_0) \geq G(\mu_0,\nu_0),
    \]
    a quick calculation results in
    \[
    0 \leq -2c\epsilon\int_\Omega g(x) d\mu_0' + \epsilon^2 \left(c^2I_K(\mu_0') -a\left(\int_\Omega U_K^{\nu_0}(x) d\mu_0'(x)\right)^2\right).
    \]
    Since $0 < a,b,c < \infty$ and we can choose $\epsilon > 0$ arbitrarily small, we must have $\int_\Omega g(x) d\mu_0'(x) = \int_{\{g > 0\}} g(x) d\mu_0'(x)\leq 0$. This implies that $\mu_0' \equiv 0$, and thus $cU_K^{\mu_0}(x) \leq a U_K^{\nu_0}(x)$ $\mu_0$-almost everywhere in $E_1$.
\end{proof}
With this lemma, we are able to apply the first maximum principle and bound $\frac{ab}{c^2}$.
\begin{proof}[Proof of Theorem \ref{thm.fmposdef}]
Since $U_K^{\mu_0}$ and $U_K^{\nu_0}$ are lower semi-continuous on $\Omega$ and continuous away from their support, $cU_K^{\mu_0}(x) \leq a U_K^{\nu_0}(x)$ and $cU_K^{\nu_0}(x) \leq bU_K^{\mu_0}(x)$ hold on all of $F_1=\supp{\mu_0}$ and $F_2=\supp{\nu_0}$, respectively. Applying the first maximum principle,
    \[
        \frac{c}{a} \sup_{x\in F_1} \leq \sup_{x \in F_1} U_K^{\nu_0}(x) \leq\sup_{x \in F_2} U_K^{\nu_0}(x) \leq\sup_{x \in F_2}\frac{b}{c} U_K^{\mu_0}(x) \leq\sup_{x \in F_1} U_K^{\mu_0}(x)
    \]
So $\frac{c}{a} \leq \frac{b}{c}$, and thus $G(\mu_0,\nu_0) \geq 1$.
    \par
Now for any signed $\mu \in \M(\Omega)$, its Jordan decomposition gives $\mu^+ - \mu^-$, where $\mu^+ = \mu\resmes P$ and $\mu^- = -\mu \resmes N$ for disjoint Borel $P$ and $N$. Since $\mu^+$ and $\mu^-$ are finite Borel measures on $\Omega$, they are regular measures. Thus, there exists sequences of compact subsets of $P$ and $N$ $\{P_n\}_{n=1}^\infty,\{N_n\}_{n=1}^\infty$ such that $\mu^+(P) = \lim_{n\to \infty} (\mu \resmes P_n)(P)$ and $\mu^-(N) = \lim_{n\to\infty} (\mu \resmes N_n)(N)$. Then monotone convergence tells us that
    \[
    G(\mu^+,\mu^-) = \lim_{n\to \infty} \inf_{\substack{\supp{\mu} \subseteq P_n\\\supp{\nu} \subseteq N_n}} G(\mu,\nu) \geq 1.
    \]
So $I_K(\mu) \geq 0$, and thus $K$ is positive definite.
\end{proof}

\section{Minimizing Measures for $G$-invariant Kernels on Compact Homogeneous Manifolds}
\label{sec.torus}
In this section, we restrict ourselves to the case where $\Omega = M$, a $d \geq 1$ dimension homogeneous space of a compact Lie group $G$ embedded in $\R^{d+r}$. We also require that $M$ is reflexive (for all $x,y \in M$, there exists $g \in G$ such that $gx =y$, $gy =x$). For this space, there exists the uniform surface measure $\sigma$ with $\sigma(M) = 1$.  We will be considering only kernels $K$ which are $G$-invariant\textemdash that is, $K(x,y) = K(gx,gy)$ for all $x,y \in M, g\in G$. Finally, we make the assumption that $I_K(\sigma) < \infty$.
\par
Some natural examples of these spaces are:
\begin{enumerate}
\renewcommand{\labelenumi}{\roman{enumi})}
\item The unit $d$-sphere, $\Omega=\mathbb{S}^{d}\subset \mathbb R^{d+1}$,
which is the orbit of any unit vector under the action of $SO(d+1)$,
the group of $d+1$ dimensional orthogonal matrices of determinant 1.

\item  The flat $d$-torus, ${\mathbb {T}}^d = (\mathbb{S}^1)^{d} \subset (\mathbb R^2)^{d}$, 
which is the orbit of the point $((1,0),(1,0),\ldots,(1,0))$ under rotation 
by $(\theta_1,\theta_2,\ldots,\theta_{d})$. 
Since ${\mathbb T}$ is just the compact quotient group, 
$\mathbb R/{2\pi}\mathbb Z$, the flat $d$-torus is the product 
group $\mathcal{G} = (\mathbb R/{2\pi}\mathbb Z)^d .$

\item The Grassmannians $Gr_k(\R^d)$, the set of $k$-dimensional linear subspaces of $\R^d$. This is the orbit of any $k$-dimensional linear subspace under $O(n)/(O(k) \times O(n-k))$, where $O$ is the general orthogonal group.

\end{enumerate}
\par
For group invariant kernels, the uniform surface measure $\sigma$ provides a approximately $K$-invariant measure with full support. This implies that any positive define kernel has $\sigma$ as a minimizer. We will show that on the flat torus $\T^d$ with its Euclidean distance, the Riesz kernels $K_s$ with $d > s \geq d-2$ have $\sigma$ as a minimizer.
\par
We also look at the expansion of the kernel in the harmonic polynomials of the homogeneous manifold. The positivity of the coefficients of this expansion is closely related to the positive definiteness of the kernel. For the torus, these polynomials correspond to the multivariate cosine series. We use our positive definiteness result to show that the cosine coefficients for functions which correspond to subharmonic kernels on the flat torus must be nonnegative.

\subsection{Positive Definiteness and Harmonic Polynomials}
Let $P_n$ be the polynomials with real coefficients in $d+r$ variables of degree $n$, restricted to $M$. Then let the harmonic polynomials of degree $n$ be $H_n = P_n \cap P_{n-1}^\perp$, where orthogonality is with respect to the inner product
\[
\langle f,g\rangle = \int_{M} fg d\sigma.
\]
We can always decompose $H_n$ into $G$-invariant subspaces $H_{n,k}$, each with an orthonormal basis $Y_{n,k}^1,...,Y_{n,k}^{d_{n,k}}$. Then if we define
\[
Q_{n,k}(x,y) = \sum_{i=1}^{d_{n,k}} Y_{n,k}^i(x)Y_{n,k}^j(y),
\]
each $Q_{n,k}$ is positive definite, and for $G$-invariant $K$,
\[
K(x,y) = \sum_{n=0}^\infty \sum_{k=1}^{\nu_n} a_{n,k} Q_{n,k}(x,y)
\]
with
\[
a_{n,k} = \frac{1}{d_{n,k}} \int_M K(x,y) Q_{n,k}(x,y) d\sigma(y).
\]
where convergence is in $L^2(M,\sigma)$.

We note that since $K,Q$, and $\sigma$ are $G$-invariant, $a_{n,k}$ does not depend on the choice of $x$. For more details on this construction, see \cite{DLS}.
\medskip
\begin{proposition} \label{prop.poscoeffs}
    If $K$ is positive definite, then $a_{n,k} \geq 0$ for all $n,k$.
\end{proposition}
\begin{proof}
    Assume there exists $a_{n,k} < 0$. Then
    \[
    \int_{M} K(x,y) Q_{n,k}(x,y) d\sigma(y) = \sum_{i=1}^{d_{n,k}}\int_M K(x,y) Y_{n,k}^i(x) Y_{n,k}^i(y) d\sigma(y) < 0
    \]
    so there exists $Y_{n,k}^i$ such that $\int_M K(x,y) Y_{n,k}^i(x) Y_{n,k}^i(y) d\sigma(y) < 0$. But then if $d\mu = Y_{n,k}^i d\sigma$, 
    \[
    I_K(\mu) = \int_{M} \int_{M} K(x,y) Y_{n,k}^i(x) Y_{n,k}^i(y) d\sigma(y)d\sigma(x) < 0
    \]
    and $K$ is not positive definite. 
    \par
    So we have proven the contrapositive of our statement.
    \end{proof}
    Since convergence is in $L^2$, we cannot in general state the converse of Proposition \ref{prop.poscoeffs}. However, the converse has been shown for certain singular kernels \cite{DG,BG}, and is an important tool for proving positive definiteness.
\subsection{The Torus}
We define the torus as $\mathbb{T}^d = \mathbb{S}^1 \times \mathbb{S}^1 \times ... \times \mathbb{S}^1$, with distance 
\[
\rho((z_1,z_2,...,z_d),(\zeta_1,\zeta_2,...,\zeta_d)) = \sqrt{d(z_1,\zeta_1)^2+d(z_2,\zeta_2)^2+...+d(z_d,\zeta_d)^2}
\]
where $d(z_i,\zeta_i)$ is the geodesic distance on $\mathbb{S}^1$. And after we multiply by a constant, this is isometric to $(\R/2\pi\Z)^d$, with distance 
\[
d(x,y) = \min_{z \in \Lambda} \{|(y-x)-z|\}
\]
where $\Lambda$ is the scaled integer lattice $(2\pi\Z)^d$. We denote the point opposite $z$ on $\mathbb{S}^1$ by $-z$.
\par
\medskip
\begin{proposition}\label{prop.torreiszsub}
    Let $\sigma$ be the uniform surface measure on $\T^d$. Then the Riesz kernel $K_{s}(x,y)$ is entirely subharmonic on $\T^d$ for $s \geq d-2$, and entirely strictly subharmonic for $s > d-2$.
\end{proposition}
\begin{proof}
    For any $x_0 \in \T^d$, choose $R = 1/2$. Then for any $x \neq x_0$, let $z$ be the point in $\Lambda$ closest to $x-x_0$. Now if we let $u(x) = \begin{cases} - \log |x| & s = 0\\ \sign(s)|x|^{-s} & s \neq 0 \end{cases}$, then for any $B(x,r) \subseteq \T^d \setminus \{x_0\}$
    \begin{align*}
        \frac{1}{\sigma(B(x,r))}\int_{B(x,r)} K_s(x_0,y) d\mu(y) &= \frac{1}{\sigma(B(x,r))}\int_{B(x,r)} u(d(x_0,y))^{-s} d\mu(y)\\
        &\geq \frac{1}{\sigma(B(x,r))}\int_{B(x,r)} u(|y-x_0-z|) d\mu(y)\\
        &\geq u(|x-x_0-z|)\\
        &= K_s(x_0,x)
    \end{align*}
    since $u(x)$ is subharmonic on $\R^d\setminus \{0\}$ for $s \geq d-2$. And $u$ is strictly subharmonic for $s > d-2$, so in this case, the last inequality is strict and $K_s$ is entirely strictly subharmonic.
    \par
\end{proof}
\medskip
\begin{theorem}\label{thm.torrieszequil}
The minimizing measures for Riesz kernels $K_s(x,y)$ on $\T^d$ can be characterized as follows:
    \begin{enumerate}[label=(\alph*)]
        \item If $s < -2$, the only minimizing measures consist of two points a diameter apart:
        $
        \mu = \frac{1}{2}(\delta_{(z_1,z_2,...,z_d)} + \delta_{(-z_1,-z_2,...,-z_d)})
        $
        \item If $s = -2$, minimizing measures are those whose projections for each dimension are points a diameter apart. That is, for each $1 \leq i \leq d$, if we define $\proj_i: \T^d \to \mathbb{S}^1$ by $\proj_i(z_1,...,z_d) = z_i$, the pushforward measure $\mu_i = \proj_i \# \mu$ is $\frac{\delta_z + \delta_{-z}}{2}$ for some $z \in \mathbb{S}^1$.
        \item If $s = d-2$, the uniform surface measure $\sigma$ is a minimizer.
        \item If $d > s > d-2$, the uniform surface measure $\sigma$ is a minimizer, and all minimizers must have full support.
    \end{enumerate}
\end{theorem}
\begin{proof}
    Since $I_{K_s}(\sigma) < \infty$ for $s < d$, parts (c) and (d) follow by applying the results of Proposition \ref{prop.torreiszsub} to Theorems \ref{thm.kinvar} and \ref{thm.fullsup}.
    \par
    We next prove part (b). For any $\mu \in \P(\T^d)$, we have
    \begin{align*}
        I_{K_{-2}}(\mu) &= \int_{\T^d} \int_{\T^d} -\rho(z,\zeta)^2 d\mu(z)d\mu(\zeta)\\
        &= -\int_{\T^d} \int_{\T^d} d(z_1,\zeta_1)^2 + d(z_2,\zeta_2)^2 + ... + d(z_d,\zeta_d)^2 d\mu(z)d\mu(\zeta)\\
        &= \int_{\mathbb{S}^1} \int_{\mathbb{S}^1} -d(z_1,\zeta_1)^2 d\mu_1(z_1) d\mu_1(\zeta_d) + ... + \int_{\mathbb{S}^1} \int_{\mathbb{S}^1} -d(z_d,\zeta_d)^2 d\mu_d(z_d) d\mu_d(\zeta_d).
    \end{align*}
    So for each $1 \leq i \leq d$, $\mu_i$ must be a minimizing measure of the geodesic Riesz energy squared on $\mathbb{S}^1$. This problem has been solved\cite{BDM}, with all minimizers being of the form $\mu_i = \frac{\delta_z + \delta_{-z}}{2}$.
    \par
    For part (a), we note that any $\mu = \frac{1}{2}(\delta_{(z_1,z_2,...,z_d)} + \delta_{(-z_1,-z_2,...,-z_d)})$ is a minimizer for $K_{-2}$. Now, for $s < 0$, if we define $K_s'(x,y) = -\left(\frac{\rho(x,y)}{\diam(\T^d)}\right)^{-s}$, we see that $K_s$ and $K_s'$ have the same minimizing measures, and that if $-r > -s$,
    \[
        K_{-r}'(x,y)  \geq K_{-s}'(x,y)
    \]
    with equality only when $x = y$ or $\rho(x,y) = \diam(\T^d)$. Thus, for any $\nu \in \P(\T^d)$, 
    \[
        I_{K_{-r}'}(\nu) \geq I_{K_{-2}'}(\nu) \geq I_{K_{-2}'}(\mu).
    \]
    So $\mu$ is a minimizer for $K_{-r}$. Next, we show that to achieve equality, we must have $\nu = \frac{1}{2}(\delta_{(\zeta_1,\zeta_2,...,\zeta_d)} + \delta_{(-\zeta_1,-\zeta_2,...,-\zeta_d)})$
    \par
    Since points in $\T^d$ on have one point a diameter away, to satisfy the first equality, any minimizing measure for $K_{-r}$ must be of the form $\omega \delta_{(\zeta_1,...,\zeta_d)} + (1-\omega) \delta_{(-\zeta_1,...,-\zeta_d)}$. 
    But to satisfy the second equality, we must have $\omega = \frac{1}{2}$.
    An example of minimizing measures for $s=-2$ and $s\leq -2$ is shown below.
    \begin{figure}[h!]
        \centering
        \includegraphics[width=0.25\linewidth]{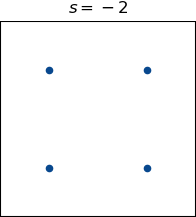}
        \includegraphics[width=0.25\linewidth]{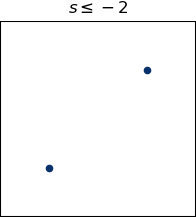}
        \label{fig:equil_measures}
        \caption{Minimizers for $s = -2$ and $s \leq -2$}
    \end{figure}
    \\
\end{proof}
This gives a solution of the energy minimization problem for the Riesz energy on the torus, when $s \in (-\infty,-2] \cup [2-d,d)$. But this technique can also be used for more general kernels.
\medskip
\begin{proposition}\label{prop.torgenequil}
    Let $f:[0,\pi]^d \to [0,\infty]$ be a function satisfy the following properties:
    \begin{enumerate}[label=(\alph*)]
        \item $f$ is continuous in the extended sense
        \item $f$ is subharmonic (and thus finite) on $[0,\pi]^d \setminus \{0\}$
        \item For each $1 \leq i \leq d$, there exists $r > 0$ such that $f(x_1,x_2,...,x_i,...,x_d) > f(x_1,x_2,...,\pi,...,x_d)$ for all $x \in [0,\pi]^d$ with $|\pi-x_i| < r$.
        \item If $f(0) = \infty$, then $f$ is Riesz equivalent around $0$. That is, there exists $c_1,c_2,r>0,s\geq 0$ such that $c_1K_s(0,y) \leq f(y) \leq c_2K_s(0,y)$ for $y \in B(0,r)$.
    \end{enumerate}
    Then if $K(x,y) = f(d(x_1,y_1),...,d(x_d,y_d))$, the uniform surface measure $\sigma$ is a minimizer for $K$ on $\T^d$.
\end{proposition}
\begin{proof} The first three conditions tell us that $K$ is entirely subharmonic on $\T^d$. And the last condition tells us that $K$ is equivalent to a Riesz kernel, which implies $K$ is regular. Thus, we can apply Theorem \ref{thm.kinvar} to see that $\sigma$ is a minimizing measure.
\end{proof}
Similar energy minimization problems have been solved by showing that the expansion of the kernel into harmonic polynomials has nonnegative coefficients, and then finding a method to show that the convergence of these polynomials implies the positive definiteness of the kernel.
\par
We have approached this problem from an entirely different direction, focusing instead on the properties of the potential of the minimizing measure. With our energy minimization results, we are able to "work backwards" and show that the coefficients of the expansion are nonnegative.
\medskip
\begin{remark}
    Another compact homogeneous manifold for which the Riesz energy problem has been studied is the sphere $\mathbb{S}^d$ with its geodesic distance. This provides a counterexample to the converse of Theorem \ref{thm.kinvar}, showing that a $K$-invariant minimizing measure does not imply subharmonicity. 
    \par
    Our example occurs when we take the Riesz kernel $K_{-1}$ on $\mathbb{S}^2$. It is known that the uniform surface measure $\sigma$ is a minimizer \cite{BDM}. But $K_{-1}$ is not entirely subharmonic. 
    For example, if we take $x,y \in \mathbb{S}^2$ such that $\rho(x,y) = \frac{\pi}{4}$, then $K(\cdot,y)$ will not satisfy the submean value inequality at $x$. This can be verified using the spherical law of cosines and the Funk-Hecke formula \cite[Lemma A.5.2]{DX}.
    \par
    The intuition for this can be seen in Figure \ref{fig:sphereical_caps}. We consider Earth, where we take $y$ to be the north pole, and $x$ to be a point in the northern hemisphere. Since lines of latitude are not great circles, the area in the cap which is north of $x$ is smaller than the area which is south of $x$.
    \begin{figure}[h!]
        \centering
        \includegraphics[width=4cm]{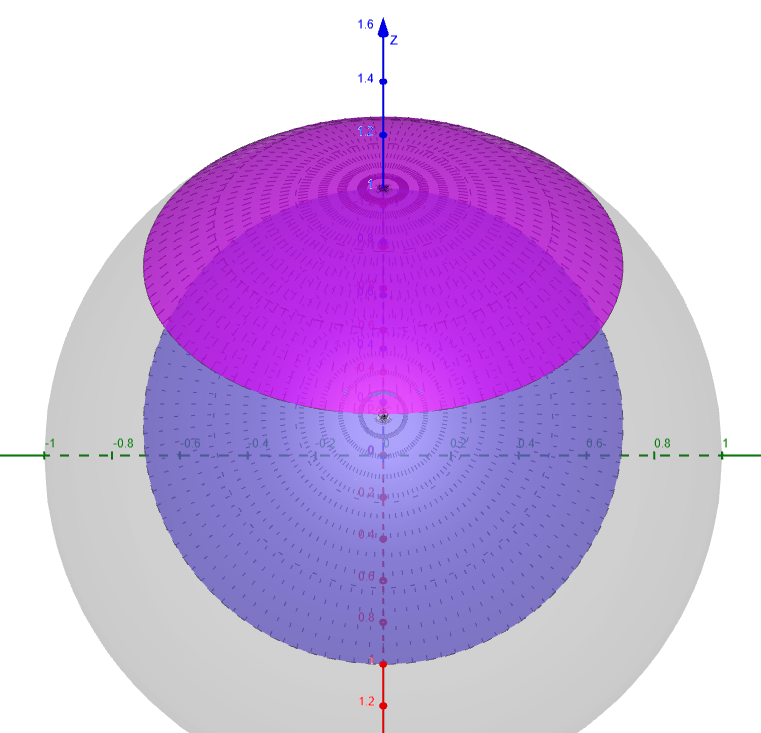}
        \caption{The overlap of spherical caps on $\mathbb{S}^2$}
        \label{fig:sphereical_caps}
    \end{figure}
\end{remark}

\subsection{Multivariate Fourier Series} By applying Proposition \ref{cor.subposdef} to the kernels in Proposition \ref{prop.torgenequil}, we see that these kernels are positive definite. We now consider the expansion of these kernels into harmonic polynomials.

On the torus $\T^d$, consider the multi-index $(\alpha_1,...,\alpha_d)$, where each integer $\alpha_i \geq 0$, and $\sum_{i=1}^d \alpha_i = n$. Then from \cite{LR} we have
\[
    Q_{n,(\alpha_1,...,\alpha_d)}((x_1,..,x_d),(y_1,...,y_d)) = \prod_{i=1}^d \cos(\alpha_i(x_i-y_i)).
\]
Since the harmonic expansion must have non negative coefficients when $\sigma$ is the minimizer,
\medskip
\begin{theorem}
    Let $f:[0,\pi]^d \to [0,\infty]$ satisfy the conditions of Proposition \ref{prop.torgenequil}. Then if all $n_i$ are nonnegative integers, the multivariate cosine Fourier coefficients are nonnegative. In other words,
    \[
     \int_{[0,\pi]^d} f(x_1,x_2,...,x_d) \cos(n_1 x_1)\cos(n_2x_2)...\cos(n_dx_d) dx_1dx_2...dx_d \geq 0.
    \]
    .
\end{theorem}
This is clear from applying Corollary \ref{cor.subposdef} and Proposition \ref{prop.poscoeffs} to the kernels in Proposition \ref{prop.torgenequil}.

It is well known that one dimensional cosine Fourier coefficients are nonnegative for convex functions. And while this property does not hold in general for convex functions of more than one variable, it has been extended to Popoviciu convex functions of two variables\cite{NR}. Here, we have a different generalization of convexity: subharmonic functions.

\section*{Declarations}
\subsection*{Funding}
    No funding was received to assist with the preparation of this manuscript.
\subsection*{Competing Interests}
    The authors have no relevant financial or non-financial interests to disclose.
\subsection*{Data Availability}
    The authors do not have any research data outside the submitted manuscript file.

\end{document}